\documentclass[11pt]{amsart}

\usepackage{fullpage}

\usepackage[english]{babel}
\usepackage[utf8]{inputenc}
\usepackage{amsmath}
\usepackage{graphicx}
\usepackage{mathtools}
\usepackage{amssymb}
\usepackage{amsthm}
\usepackage{tikz-cd}
\usepackage{mathrsfs}
\usepackage[colorinlistoftodos]{todonotes}
\usepackage{enumitem}
\usepackage{yfonts}
\usepackage{dsfont}
\usepackage{derivative}
\usepackage{breqn}
\usepackage{tikz}
\usepackage{fullpage}

\usepackage{hyperref}
\usepackage{cleveref}
\usepackage{moreenum}
\usepackage{aliascnt}

\hypersetup{
    colorlinks=true,
    linkcolor=blue,
    filecolor=black,      
    urlcolor=black,
    citecolor=red
}

\newcommand{\on}{\operatorname}

\newcommand{\Don}{\on{Don}}
\newcommand{\pr}{\on{pr}}

\theoremstyle{plain}
\newtheorem{thm}{Theorem}[section]

\newtheorem*{thm*}{Theorem}

\newaliascnt{prop}{thm}
    \newtheorem{prop}[prop]{Proposition}
    \aliascntresetthe{prop}

\newaliascnt{lem}{thm}
    \newtheorem{lem}[lem]{Lemma}
    \aliascntresetthe{lem}

\newaliascnt{cor}{thm}
    
    \aliascntresetthe{cor}

\newaliascnt{conj}{thm}
    
    \aliascntresetthe{conj}

\theoremstyle{definition}

\newaliascnt{eg}{thm}
    \newtheorem{eg}[eg]{Example}
    \aliascntresetthe{eg}

\newaliascnt{defn}{thm}
    \newtheorem{defn}[defn]{Definition}
    \aliascntresetthe{defn}

\newaliascnt{rmk}{thm}
    \newtheorem{rmk}[rmk]{Remark}
    \aliascntresetthe{rmk}

\numberwithin{equation}{section}

\newcommand{\R}{\mathbb{R}}
\newcommand{\C}{\mathbb{C}}

\newcommand{\D}{\mathfrak{D}}
\newcommand{\U}{\mathfrak{U}}

\usetikzlibrary{positioning}
\usepackage{caption}
\usepackage{pgfplots}
\pgfplotsset{compat=1.17}
\usetikzlibrary{intersections}
\usetikzlibrary{patterns}
\usetikzlibrary{arrows}
\usetikzlibrary{decorations.markings}
\usepgfplotslibrary{fillbetween}
\usepackage{pgfplots}

\begin{document}

\title{The focus-focus addition graph is immersed}

\author{Mohammed Abouzaid, Nathaniel Bottman, and Yunpeng Niu}
\date{}

\begin{abstract}
For a symplectic 4-manifold $M$ equipped with a singular Lagrangian fibration with a section, the natural fiberwise addition given by the local Hamiltonian flow is well-defined on the regular points.
We prove, in the case that the singularities are of focus-focus type, that the closure of the corresponding addition graph is the image of a Lagrangian immersion in $(M \times M)^- \times M$, and we study its geometry.
Our main motivation for this result is the construction of a symmetric monoidal structure on the Fukaya category of such a manifold.
\end{abstract}

\maketitle

\setcounter{tocdepth}{1}
\tableofcontents

%--------------------------------------------------------------------------%
\section{Introduction}

Homological Mirror Symmetry \cite{kon95} predicts an equivalence between the Fukaya category of a symplectic manifold and the derived category of coherent sheaves on a mirror algebraic variety, when the latter exists.
Under this equivalence, structures on the derived category of coherent sheaves give rise to structures on the Fukaya category.
In particular, as the derived category of coherent sheaves on an algebraic variety is equipped with a tensor product, there is a corresponding monoidal structure on the Fukaya category of the mirror symplectic manifold.
We intend to give a geometric construction of this monoidal structure in a class of examples motivated by SYZ mirror symmetry \cite{SYZ96}.
This project will comprise the current paper, as well as \cite{ABN24}, \cite{Niu24}, and possibly additional work.
In the current paper, we take a first step, by studying the geometry of the Lagrangian correspondence at the heart of this construction.

A smooth Lagrangian fibration is locally a completely integrable Hamiltonian system after a choice of local trivialization on the base.
The Hamiltonian flows induce a fiberwise addition when choosing a local Lagrangian section as the origin for time coordinates.
For a symplectic manifold $(M,\omega)$ equipped with a smooth Lagrangian fibration with (Lagrangian) section, there is thus a natural group structure on each torus fiber defined independent of coordinates.
In 2010, Subotic \cite{Sub10} proved that the corresponding addition graph 
\begin{equation}
    \Gamma\coloneqq\{(x,y,z)\in M^3\;\;|\;\;z=x+y\}
\end{equation}
is Lagrangian with respect to the symplectic form $(-\omega)\oplus(-\omega)\oplus\omega$, and is thus a smooth Lagrangian correspondence from $M \times M$ to $M$.

In \cite[Theorem 1.0.1]{Sub10}, Subotic used Wehrheim--Woodward's quilted Floer theory to associate to $\Gamma$ a monoidal structure on the extended Donaldson--Fukaya category of $M$, assuming that it satisfies geometric hypotheses sufficient for all the objects below to be well-defined.
Specifically, Subotic defined a tensor product $\hat\otimes$ as the composition
\begin{align}
\Don^\#(M)\times\Don^\#(M)
\stackrel{i}{\longrightarrow}
\Don^\#(M\times M)
\stackrel{\Phi^\#_\Gamma}{\longrightarrow}
\Don^\#(M),
\end{align}
where $i$ is a fully faithful functor that extends the product functor on $\Don(M)$, and $\Phi^\#_\Gamma$ is the functor associated to the correspondence $\Gamma$ as in \cite{WW:functoriality}.
In the case that $M$ is a $2$-torus, Subotic checked that this monoidal structure agrees with the one on the derived category of the mirror elliptic curves, under the mirror equivalence constructed in \cite{Zas98}.

In this paper, we generalize Subotic's Lagrangian correspondence construction to the first singular case.
We work with a symplectic 4-manifold $(M,\omega)$ equipped with a smooth proper map $\pi\colon M \rightarrow B$ onto a smooth surface $B$, with connected fibers, such that $\pi$ is everywhere a submersion except for finitely-many points that lie on pairwise distinct fibers.
We require the fibration $\pi$ to be Lagrangian with a (Lagrangian) section, i.e.\ we assume that $\omega$ restricts to zero on the regular part of each fiber and there exists a smooth section $\sigma\colon B \rightarrow M$ such that $\sigma^*\omega=0$.
Finally, we require the singular points to be of focus-focus type (see \autoref{subsec:focus-focus}).
This is the symplectic version of $A_1$-type singular fibers in elliptic fibrations (in fact they are locally smooth equivalent), where the regular torus fiber degenerates to a pinched torus.
This is a special case of almost toric symplectic manifold introduced by Symington in \cite{Symington01}.

We prove in this case that the addition, which is well-defined on the regular points, can be extended to pairs of points, assuming at most one of which is a singular point.
Such addition gives as before an addition graph $\Gamma\subset (M\times M)^-\times M$.
The main result in this paper is the following.
\begin{thm*}[\autoref{thm:Gamma-bar_is_immersed_Lagrangian}]
The closure $\bar{\Gamma}$ of the addition graph $\Gamma$ is an immersed Lagrangian submanifold of $(M \times M)^- \times M$.
It has transversal double points at each triple $(s,s,s)$, where $s$ is a critical point.
\end{thm*}

\noindent
Moreover, $\bar\Gamma$ is symmetric and associative, where ``symmetric'' means that $\tau(\bar\Gamma) = \bar\Gamma$ for $\tau$ the map sending $(x,y,z)$ to $(y,x,z)$, and where ``associative'' means that the geometric compositions $\bar\Gamma \circ (\bar\Gamma\times\Delta_M)$ and $\bar\Gamma\circ(\Delta_M\times\bar\Gamma)$ agree.
Indeed, these statements are evident for $\Gamma$, which implies the statements for $\bar\Gamma$.

In subsequent papers \cite{Niu24} and \cite{ABN24}, we will use this immersed Lagrangian correspondence to construct a monoidal structure on the Fukaya category of $M$ in various cases and derive various properties of the functor.

\subsection{Outline}

The arrangement of this paper is the following.
In \autoref{subsec:focus-focus}, we review the definition of focus-focus singularities and discuss various basic results concerning them.
This includes a key technical result \autoref{thm:focus-focus_classification}, which we will use crucially for the proof of the main theorem.
In \autoref{sec:group_structure}, we analyze the group structure in local charts and explain how it resembles the case for $A_1$-singularities of elliptic fibrations.
In \autoref{sec:addition_graph}, we first prove the main theorem \autoref{thm:Gamma-bar_is_immersed_Lagrangian} in \autoref{subsec:Gamma_is_immersed}, then in \autoref{subsec:geometry_of_correspondence} we study the local geometric features of this immersed submanifold and its minimal resolution.

\subsection*{Acknowledgements}

M.A.\ is supported by NSF award DMS-2103805.

N.B.\ thanks the Max Planck Society for its support.

Y.N.\ was partially supported by the Simons Foundation International.
She is grateful to her advisor, Kenji Fukaya, for suggesting this interesting problem as part of her thesis project and for his support, guidance, and various useful conversations during the process. 

\section{Background}
\label{sec:background}

\subsection{Focus-focus singularities}
\label{subsec:focus-focus}

In this section, we review several relevant results concerning focus-focus singularities.
Our primary reference is \cite{Evans22}.

Consider the following Hamiltonian system: 
\begin{align} \label{eq:model_Hamiltonian_focus-focus}
    (\C^2,\omega_0=dp_1\wedge dq_1+dp_2\wedge dq_2) &\stackrel{H}{\longrightarrow} \C, \\
    (p,q) &\mapsto -\bar{p}q=H_1+iH_2. \nonumber
\end{align}
This map has a critical point at $(0,0)$.
The Hamiltonian flows are 
\begin{align}
\label{eq:hamiltonian_flows}
    \phi_{H_1}^t\colon (p,q) &\mapsto(e^tp,e^{-t}q), \\
    \phi_{H_2}^t\colon (p,q) &\mapsto (e^{it}p,e^{it}q). \nonumber
\end{align}
The Hamiltonian flow of $H_2$ is always periodic and the period stays constant.
To see the Hamiltonian flow of $H_1$, we can consider the projection of the above map to the $(|p|,|q|)$-plane, exhibited in \autoref{fig:focus-focus-sing}.

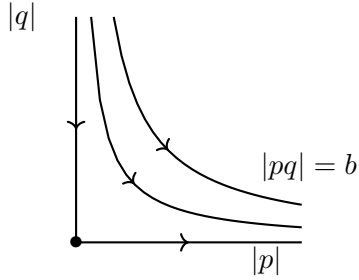
\begin{figure}
\begin{tikzpicture}[decoration={
    markings,
    mark=at position 0.5 with {\arrow{>}}}]
\node[label=left:$ \left| p\right|$] at (3,-0.25) {};
\node[label=left:$ \left| q\right|$] at (-0.25,3) {};
\draw[thick,postaction={decorate}] (0,3) -- (0,0);
\draw[thick,postaction={decorate}] (0,0) -- (3,0);
\draw[thick,domain=0.5:3.0,postaction={decorate}] plot (\x,{1.5/(\x)});
\draw[thick,domain=0.2:3.0,postaction={decorate}] plot (\x,{0.6/(\x)});
\node at (0,0) {\(\bullet\)};
\node[label=left:${\left| pq\right|=b}$] at (4,1) {};
\end{tikzpicture}
\caption{
\label{fig:focus-focus-sing}
Hamiltonian flow of $H_1$
}
\end{figure}

On a regular fiber (i.e.\ at a point where $-\bar{p}q=b \neq 0$), the $H_1$-flow is periodic.
On the singular fiber $pq=0$, the $H_1$-flow has three disjoint orbits: the positive $|p|$-axis (i.e.\ $\{|p| > 0\}$), the positive $|q|$-axis, and the singular point $(0,0)$.
This is the \textbf{standard focus-focus system}.

If $M^4$ is a symplectic 4-manifold and $B^2$ is a smooth surface, we say a map  $\pi\colon (M^4,\omega) \to B^2$ has a \textbf{focus-focus singularity at $x$} if there exists a Darboux neighbourhood of $x$ on which $\pi$ is given by \autoref{eq:model_Hamiltonian_focus-focus}, with respect to a choice of coordinates on $B^2$.
Note that this can also be regarded as the local model for a Lefschetz fibration.
Since we assume that our Lagrangian fibration has compact and connected fibers, the singular fiber will have to be a pinched torus if it has only one singular point and the singular point is of focus-focus type.
We thus refer to such a singular fiber as being of \textbf{symplectic $A_1$-type}.

Each singular fiber $F$ admits a neighborhood in $M$ which does not contain any other singular fiber.
Choosing a local trivialization on $B$ around $b\coloneqq\pi(F)$, and shrinking the neighborhood if necessary, we can get an integrable Hamiltonian system with only one singular point and of focus-focus type.
We call the obtained Hamiltonian system a \textbf{focus-focus neighborhood} when the projection to $B$ is a disc.
Any two focus-focus neighborhood are fiberwise diffeomorphic, thus there are no nontrivial differential invariants.
However, there is a nontrivial symplectic invariant: we can associated to any focus-focus neighborhood a smooth function $S$ defined in a neighborhood of the critical value $b$, which vanishes at $b$, and for which one has the following classification result by V\~u Ng\d{o}c \cite{Ngoc02}.

\begin{thm}
\label{thm:focus-focus_classification}
The germ $(S)_\infty$ of $S$ at the origin classifies the focus-focus neighborhood up to fibered symplectomorphism.
\end{thm}

Any given $S$ can be realized as the invariant of a \textbf{model neighborhood} by an explicit construction starting with the standard focus-focus system.
\begin{eg}
Given a function $S$ on the $2$-disc, denote its two partial derivatives as $S_1=\partial_1S$ and $S_2=\partial_2S$.
Take the subset 
\begin{equation}
    X\coloneqq\{(p,q)\in \C^2\; : \; |pq| < \epsilon\}
\end{equation}
in the standard focus-focus system.
We can again use projection to $|p|,|q|$-plane to illustrate this, as in \autoref{fig:model_neighborhood} below.

\begin{figure}
\begin{tikzpicture}[decoration={
    markings,
    mark=at position 0.5 with {\arrow{>}}}
    ]
\filldraw[fill=lightgray,opacity=0.5,draw=none,domain=1.5:4.0] (0,4) -- (1.5,4) plot (\x,{6/(\x)}) to[out=-45,in=45] (4,0.75) to[out=-135,in=135] (4,0) -- (0,0) -- (0,4);
\filldraw[fill=lightgray,draw=none,domain=1.5:1.8] (0,4) -- (1.5,4) plot (\x,{6/(\x)}) -- (0,6/1.8) -- (0,4);
\filldraw[fill=lightgray,draw=none,domain=1.5:1.8] (4,1.5) to[out=-45,in=45] (4,0.75) to[out=-135,in=135] (4,0) -- (5,0) to[out=135,in=-135] (5,0.6) to[out=45,in=-45] (5,1.2) to[out=170,in=-10] (4,1.5);
\draw[thick,domain=1.3:5.5] plot (\x,{6/(\x)});
\draw[->,thick] (0,0) -- (0,5);
\draw[->,thick]  (0,0) -- (6,0);
\draw[thick] (0,4) -- (1.5,4);
\node at (-0.7,4) [above] {};
\draw[thick] (4,1.5) to[out=-45,in=45] (4,0.75) to[out=-135,in=135] (4,0);
\node at (4.5,4.5) [above right] {};
\node at (1.5,1.5) {};
\node[label=left: $\sigma_1$] at (4.5,-0.25) {};
\node[label=left: $\sigma_2$] at (-0.002,4) {};
\node[label=left:$ \left| p\right|$] at (6.5,-0.5) {};
\node[label=left:$ \left| q\right|$] at (-0.5,5) {};
\node[label=left:${\left| pq\right|=\epsilon}$] at (6,1.75) {};
\end{tikzpicture}
\caption{
\label{fig:model_neighborhood}
The model neighborhood.
}
\end{figure}
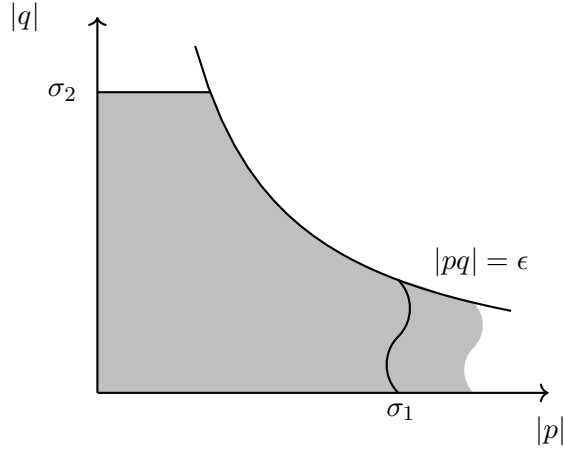

Consider the following Lagrangian sections: 
\begin{align}
    \sigma_0(b) & \coloneqq (1,-b), \\
    \sigma_1(b) &\coloneqq \bigl(\phi_{H_1}^{S_1} \circ \phi_{H_2}^{S_2}\bigr)(\sigma_0)=\bigl(e^{S_1(b)+iS_2(b)},-e^{-S_1(b)+iS_2(b)}b\bigr),
    \nonumber \\
    \sigma_2(b) & \coloneqq(-\bar{b},1).
    \nonumber
\end{align}
Note that $\sigma_1$ is a Lagrangian section because $\partial_1S_2=\partial_2S_1$, which is automatically satisfied since $S$ is a smooth function.
Now let $\sigma_1$ and $\sigma_2$ flow by time $t=(t_1,t_2)\in [0,\delta) \times [0,2\pi)$ (this corresponds to the dark shaded area in \autoref{fig:model_neighborhood}).
Consider 
\begin{equation}
    X'=\{(p,q)\in X\;\;|\;\;|q| \leq 1,\;|p| < e^{S_1(-\bar{p}q)+\delta}\},
\end{equation}
which is represented by the total shaded area in \autoref{fig:model_neighborhood}, and form the quotient space $X_S\coloneqq X'/\sim$ by identifying $\bigl(\phi_{H_1}^{t_1} \circ \phi_{H_2}^{t_2}\bigr)(\sigma_1(b)) \sim \bigl(\phi_{H_1}^{t_1} \circ \phi_{H_2}^{t_2}\bigr)(\sigma_2(b))$ for any $t=(t_1,t_2)\in [0,\delta) \times [0,2\pi)$ and $|b| < \epsilon$ (this identifies the two separate dark shaded areas in \autoref{fig:model_neighborhood} using Hamiltonian flow).
As this identifies the local Liouville coordinates, it is a symplectomorphism.
Thus, the symplectic form descends to $X_S$, and so does the map $H$. Since construction identifies $\sigma_1$ and $\sigma_2$, they together descend to a Lagrangian section defined on $X_S$, we will denote it as $\sigma_S\colon D_{\epsilon} \rightarrow X_S$.

Note that in our resulting $X_S$, on a regular fiber it takes time
\begin{equation}
    t(b)=(S_1(b)-\ln|b|,S_2(b)+\arg(b)-\pi)
\end{equation}
for $\sigma_2$ to flow to $\sigma_1$, i.e.\ for $\sigma_S$ to go through a period.
\null\hfill$\triangle$
\end{eg}

From this one can also see that given a focus-focus neighborhood, the associated invariant $S$ is given by the travel time along the local Hamiltonian flow from $\sigma_0$ to $\sigma_2$ outside of the focus-focus chart.

We will denote this Hamiltonian system as $H\colon X_S \to D_{\epsilon} \subset \C=\R^2$ and call it the \textbf{model neighborhood associated to the invariant $S$}.
In the following, we will denote its singular fiber as $F_s$ and its singular point (i.e.\ the origin) as $s$.

\subsection{Group structure}
\label{sec:group_structure}

We first define the symplectic group structure on a smooth Lagrangian submersion.
With a choice of Lagrangian section of a smooth Lagrangian submersion, we can always define good coordinates, called Liouville coordinates, by choosing an arbitrary local trivialization on the base.

\begin{defn}
Let $H\colon (M,\omega) \rightarrow \mathbb{R}^{n}$ be a complete integrable Hamiltonian system, let $B \subseteq \mathbb{R}^{n}$ be an open set, and let $\sigma\colon B \rightarrow M$ be a local Lagrangian section.
Define
\begin{equation}
    \Psi\colon B \times \mathbb{R}^{n} \rightarrow M,
    \quad
    \Psi(b,t)\coloneqq\phi_{H}^{t}(\sigma(b)).
\end{equation}
Then \(\Psi\) is a local diffeomorphism and $\Psi^{*} \omega=\sum d b_{i} \wedge d t_{i}$, where \(\left(b_{1}, \ldots, b_{n}\right)\) are the standard coordinates on \(B \subseteq \mathbb{R}^{n}\).
This coordinate system is called \textbf{Liouville coordinates}.
\null\hfill$\triangle$
\end{defn}

\begin{rmk}
It is possible to give a coordinate-free version of the Liouville coordinates, as arising from a map $T^* B \to M$.
We work in coordinates because this will make the discussion of the local model near focus-focus singularities clearer.
\null\hfill$\triangle$
\end{rmk}

Note that only when the section is Lagrangian can the pullback symplectic form become standard.
Using a choice of Liouville coordinates, we can define the operation of fiberwise addition.

\begin{defn}
Given two points $x,y\in F_b$ together with a choice of local trivialization on the base, define the \textbf{symplectic group structure} by
\begin{equation}
    (x+y)_{\sigma} \coloneqq \phi_{H}^{t+t'}(\sigma(b)),
\end{equation}
where the coordinates $t$ and $t'$  are such that $\phi_{H}^{t}(\sigma(b))=x$ and $\phi_{H}^{t'}(\sigma(b))=y$.
\null\hfill$\triangle$
\end{defn}

This group structure is independent of the choice of local trivialization.
Indeed, two different choices of local trivialization will differ by a base diffeomorphism, which will result in a linear transformation between their $t$ coordinates.
Since this transformation only depends on the base, i.e.\ remains constant on each fiber, the group structure remains the same.
The symplectic group structure is Abelian with $\sigma$ as the identity element.

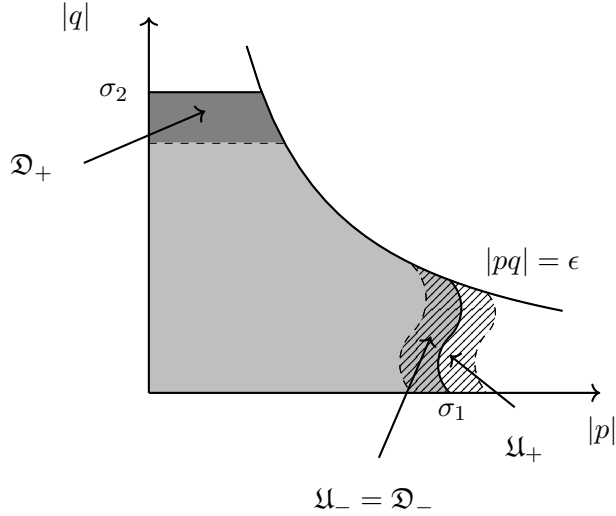
\begin{figure}
\begin{tikzpicture}[decoration={
    markings,
    mark=at position 0.5 with {\arrow{>}}}
    ]
\filldraw[fill=lightgray,opacity=0.5,draw=none,domain=1.5:4.0] (0,4) -- (1.5,4) plot (\x,{6/(\x)}) to[out=-45,in=45] (4,0.75) to[out=-135,in=135] (4,0) -- (0,0) -- (0,4);
\filldraw[fill=gray,opacity=0.5,draw=none,domain=1.5:1.8] (0,4) -- (1.5,4) plot (\x,{6/(\x)}) -- (0,6/1.8) -- (0,4);
\filldraw[fill=white,draw=none,domain=1.5:1.8] (4,1.5) to[out=-45,in=45] (4,0.75) to[out=-135,in=135] (4,0) -- (5,0) to[out=135,in=-135] (5,0.6) to[out=45,in=-45] (5,1.2) to[out=170,in=-10] (4,1.5);
\draw[thick,domain=1.3:5.5] plot (\x,{6/(\x)});
\draw[dashed] (0,3.32) -- (1.8,3.32);
\draw[->,thick] (0,0) -- (0,5);
\draw[->,thick]  (0,0) -- (6,0);
\draw[thick] (0,4) -- (1.5,4);
\node at (-0.7,4) [above] {};
\draw[thick] (4,1.5) to[out=-45,in=45] (4,0.75) to[out=-135,in=135] (4,0);
\node at (4.5,4.5) [above right] {};
\draw[dashed] (3.5,1.7) to[out=-45,in=45] (3.5,0.8) to[out=-135,in=135] (3.5,0);
\draw[dashed] (4.5,1.32) to[out=-45,in=45] (4.5,0.8) to[out=-135,in=135] (4.5,0);
\node at (4.5,4.5) [above right] {};
\node at (1.5,1.5) {};
\node[label=left: $\sigma_1$] at (4.5,-0.25) {};
\node[label=left: $\sigma_2$] at (-0.002,4) {};
\node[label=left:$ \left| p\right|$] at (6.5,-0.5) {};
\node[label=left:$ \left| q\right|$] at (-0.5,5) {};
\node[label=left:${\left| pq\right|=\epsilon}$] at (6,1.75) {};
\path[name path=A, draw, dashed] (3.5,1.7) to[out=-45,in=45] (3.5,0.8) to[out=-135,in=135] (3.5,0);
\path[name path=B, draw, dashed] (4.5,1.32) to[out=-45,in=45] (4.5,0.8) to[ out=-135,in=135] (4.5,0);
\path[name path=C, draw] (0,0) -- (6,0);
\path[name path=D, draw,domain=1.3:5.5] plot (\x,{6/(\x)});
\tikzfillbetween[of=A and B]{pattern=north east lines};
\node[label=left:$\mathfrak{D}_+$] (S) at (-1,3) {};
\draw[->, thick] (S) -- (0.75,3.75);
\node[label=below:${\mathfrak{U}_-=\mathfrak{D}_-}$] (T) at (3,-1) {};
\draw[->, thick] (T) -- (3.75,0.75);
\node[label=below:$\mathfrak{U}_+$] (R) at (5,-0.3) {};
\draw[->, thick] (R) -- (4,0.5);
\end{tikzpicture}
\caption{
\label{fig:charts_for_X_S}
Charts for $X_S$.
}
\end{figure}

We now examine the symplectic group structure in our model neighborhood $X_S$.
For the sake of later discussion, we will use the following charts to describe $X_S$, as illustrated in \autoref{fig:charts_for_X_S}.
These charts are
\begin{align}
    \D^0 & \coloneqq \{(p,q)\in X\;\;|\;\;|q|< 1,|p|< e^{S_1(-\bar{p}q)}\}, \\
    \mathfrak{U} &\coloneqq \{(p,q)\in X\;\;|\;\; e^{S_1(-\bar{p}q)-\delta}<|p|<e^{S_1(-\bar{p}q)+\delta}\}, \nonumber 
\end{align}
where $\D^0$ is the entire shaded region from $\sigma_1$ to $\sigma_2$ in \autoref{fig:charts_for_X_S} and $\U$ is the dashed region.
We have maps
$\phi_0\colon \D^0 \hookrightarrow X \rightarrow X_S$
and $\phi_1\colon \U \hookrightarrow X \to X_S$ that are diffeomorphisms onto open subsets of $X_S$ since they embed into $X$, then descend to $X_S$.
Note that their intersection in $X_S$ is the disjoint union of two open subsets
\begin{equation}
    \phi_0(\D^0)\cap \phi_1(\U)=V_- \cup V_+,
\end{equation}
where
\begin{align}
    \phi_0^{-1}(V_+) &\coloneqq\D_+\coloneqq\{(p,q)\in X\;\;|\;\; e^{-\delta}<|q|<1\}, \\
    \phi_1^{-1}(V_+)&\coloneqq\U_+\coloneqq\{(p,q)\in X\;\;|\;\; e^{S_1(-\bar{p}q)}<|p|<e^{S_1(-\bar{p}q)+\delta}\}, \nonumber \\
    \phi_0^{-1}(V_-)&\coloneqq\D_-\coloneqq\{(p,q)\in X\;\;|\;\; e^{S_1(-\bar{p}q)-\delta}<|p|<e^{S_1(-\bar{p}q)}\}, \nonumber \\
    \phi_1^{-1}(V_-)&\coloneqq\U_-\coloneqq\{(p,q)\in X\;\;|\;\; e^{S_1(-\bar{p}q)-\delta}<|p|<e^{S_1(-\bar{p}q)}\}. \nonumber
\end{align}
The transition functions between the two charts are given by
\begin{align}
    (p,q)& \in \D_+ \xmapsto{\phi_0 \circ \phi_1^{-1}} (e^{S_1(-\bar{p}q)+iS_2(-\bar{p}q)}/\bar{q},\;\bar{p}q^2e^{-S_1(-\bar{p}q)+iS_2(-\bar{p}q)})\in\U_+, \\
    (p,q) & \in \D_- \xmapsto{\phi_0 \circ \phi_1^{-1}} (p,q) \in \U_-, \nonumber \\
    (p,q) &\in \U_+ \xmapsto{\phi_1 \circ \phi_0^{-1}} (p^2\bar{q}e^{-S_1(-\bar{p}q)-iS_2(-\bar{p}q)},\;e^{S_1(-\bar{p}q)-iS_2(-\bar{p}q)}/\bar{p})\in\D_+, \nonumber \\
    (p,q) & \in \U_- \xmapsto{\phi_1 \circ \phi_0^{-1}} (p,q) \in \D_-. \nonumber
\end{align}

In order to take care of the quotient, we call the region given by
\begin{equation}
    \D\coloneqq\{(p,q)\in X\;\;|\;\;|q|< 1,\; |p|\leq e^{S_1(-\bar{p}q)}\}
\end{equation}
\textbf{the formal domain}.
Note that $\D$ can be identified with $X_S$ set-theoretically.

If we take out the singular fiber $F_s$, we obtain a smooth Lagrangian submersion
\begin{equation}
    H|_{X_S\backslash F_s}\colon X_S\backslash F_s \to D_{\epsilon}\backslash \{0\}\subset \R^2,
\end{equation}
where $D_{\epsilon}$ is the open disk with radius $\epsilon$ in $\C$ centered at the origin.
By choosing $\sigma_S$ as our Lagrangian section, we can define a group structure on the regular fibers (see \autoref{r4.3} for discussion on a different choice of Lagrangian section).

We may describe the addition formula in $\D$ as follows: using \autoref{eq:hamiltonian_flows}, in $X$, when we use $\sigma_1$ as the origin, addition takes the form
\begin{equation}
\label{Add1}
    (p_1,q_1)+(p_2,q_2)=(e^{-S_1(b)-iS_2(b)}p_1p_2,\;e^{S_1(b)-iS_2(b)}q_1\slash \bar{p}_2).
\end{equation}
When we choose $\sigma_2$ as the origin, addition takes the form
\begin{equation}
\label{Add2}
    (p_1,q_1)+(p_2,q_2)=(p_1\slash\bar{q}_2,q_1q_2),
\end{equation}
where $-\bar{p_1}q_1=-\bar{p_2}q_2=b \neq 0$.
By requiring all of the points involved to lie in $\D$, we can get a genuine addition on $X_S$ after identification.
The following result asserts that addition on $X_S\backslash F_s$ is entirely described by the two formulae in the following lemma.

\begin{lem}
Given two arbitrary points $x,y\in X_S$ with coordinates $x=(p_1,q_1), y=(p_2,q_2)$ in $\D$, we have 
\begin{align}
    (x+y)_{\sigma_1} & \in \D \iff |e^{-S_1(b)-iS_2(b)}p_1p_2| > |b| \iff |q_1q_2| < e^{-S_1(b)}|b|, \\
    (x+y)_{\sigma_2} & \in \D \iff |q_1q_2| \geq e^{-S_1(b)}|b|. \nonumber
\end{align}
\end{lem}

\begin{proof}
The above inequalities can be obtained by direct computation. In fact, take two points that are close to $\sigma_1$ in $\D$ and consider $(x+y)_{\sigma_2}$ (see \autoref{fig:addition_in_D}).
By the counter-clockwise direction of the Hamiltonian flow of $H_1$, both $x$ and $y$ are too ``positive'' in $t_1$ coordinates for $(x+y)_{\sigma_2}$ to lie in $\D$.
In order for it to formally lie in $\D$, we can flow $(x+y)_{\sigma_2}$ backwards in the $t_1$ direction for a full period; i.e., we consider $\phi_{H_1}^{-S_1(b)+ln|b|}((x+y)_{\sigma_2})$.
This point will definitely lie in $\D$, thus defining a valid point under the identification of $X_S$ and $\D$.
This procedure has the same effect as formally switching to $\sigma_1$ and using $(x+y)_{\sigma_1}$ as the addition point instead.

Thus, for any $x,y \in \D$, only one of the points $(x+y)_{\sigma_1}$ and $(x+y)_{\sigma_2}$ can be in $\D$ since they differ by a period in $t_1$.
Also, one of them has to lie in $\D$ since they cannot be away from $\D$ for more than one period in the $t_1$ coordinate.
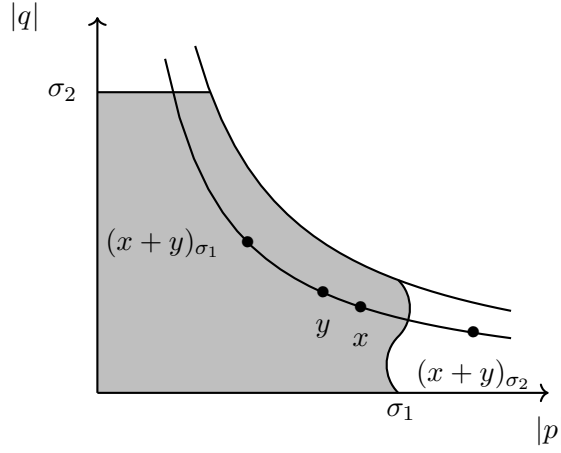
\begin{figure}
\begin{tikzpicture}[decoration={
    markings,
    mark=at position 0.5 with {\arrow{>}}}
    ]
\filldraw[fill=lightgray,opacity=0.5,draw=none,domain=1.5:4.0] (0,4) -- (1.5,4) plot (\x,{6/(\x)}) to[out=-45,in=45] (4,0.75) to[out=-135,in=135] (4,0) -- (0,0) -- (0,4);
\draw[thick,domain=1.3:5.5] plot (\x,{6/(\x)});
\draw[thick,domain=0.9:5.5] plot (\x,{4/(\x)});
\draw[->,thick] (0,0) -- (0,5);
\draw[->,thick]  (0,0) -- (6,0);
\draw[thick] (0,4) -- (1.5,4);
\node at (-0.7,4) [above] {};
\draw[thick] (4,1.5) to[out=-45,in=45] (4,0.75) to[out=-135,in=135] (4,0);
\node at (4.5,4.5) [above right] {};
\node at (1.5,1.5) {};
\node[label=left: $\sigma_1$] at (4.5,-0.25) {};
\node[label=left: $\sigma_2$] at (-0.002,4) {};
\node[label=left:$ \left| p\right|$] at (6.5,-0.5) {};
\node[label=left:$ \left| q\right|$] at (-0.5,5) {};
\node[label=below: $x$] at (3.5,1.14) {\(\bullet\)};
\node[label=below: $y$] at (3,1.33) {\(\bullet\)};
\node[label=below: ${(x+y)_{\sigma_2}}$] at (5,0.8) {\(\bullet\)};
\node[label=left: ${(x+y)_{\sigma_1}}$] at (2,2) {\(\bullet\)};
\end{tikzpicture}
\caption{
\label{fig:addition_in_D}
Addition of points in $\D$.
}
\end{figure}
\end{proof}

Now we consider the induced ``group structure" on the singular fiber by considering the limit of the above equations on $\{pq=0\} \cap \D$.

\begin{prop}\label{p3.5}
The symplectic group structure extends to a smooth map
\begin{equation} \label{eq:fiberwise_addition_inclusing_singular_points}
    \bigl(X_S\times_{D_\epsilon}X_S\bigr)\setminus\{(s,s)\} \to X_S,
\end{equation}
over the projection to $D_\epsilon$.
\end{prop}

\begin{proof}
Consider first the case where $x,y$ are both regular points on the singular fiber; that is, $|p_1|^2+|q_1|^2, |p_2|^2+|q_2|^2 \neq 0$.
By allowing one of the coordinates to be zero in \autoref{Add1} and \autoref{Add2}, we formally obtain
\begin{align}
    (p_1,0)+(p_2,0) & =(p_1p_2e^{-S_1(0)-iS_2(0)},0), \\
    (0,q_1)+(0,q_2) & =(0,q_1q_2), \nonumber \\
    (p,0)+(0,q) & =
    \begin{cases}
        (p\slash\bar{q},0), & |p\slash q|\leq e^{S_1(0)}, \\ 
        (0,e^{S_1(0)-S_2(0)}q \slash \bar{p}), & |p\slash q|> e^{S_1(0)}.
    \end{cases} \nonumber
\end{align}

Thus, the group structure on the regular fiber induces a well-defined Abelian group structure on the regular part of the singular fiber.
Next, allowing one of the two points to be the singular point (i.e.\ $(0,0) \in \D$), we get
\begin{align}
    (p,0)+(0,0) &= (0,0), \\
    (0,q)+(0,0) &= (0,0). \nonumber
\end{align}
This tells us that the singular point plus any regular point will give the singular point.
So far we have shown that addition can be defined for $F_s \times F_s - \{(s,s)\}$.
\end{proof}

\begin{rmk}
In fact, \autoref{eq:fiberwise_addition_inclusing_singular_points} defines, as expected, a group structure on the regular part of the singular fiber, making it isomorphic to $(\C^*,\times)$, under the following explicit identification
\begin{equation}
    \C^* \rightarrow F_s, \;\;\;z\in \C^* \mapsto
    \begin{cases}
        (ze^{S_1(0)+iS_2(0)},0) & \text{if} \quad |z|\leq 1\\ 
        (0,1/\bar{z}) & \text{if} \quad |z|> 1.
    \end{cases}.
\end{equation}
Here in the formula we identifies $F_s$ with $\{pq=0\} \cap \D$. Under this identification, the whole singular fiber corresponds to the result of identifying $0$ and $\infty$ in $\C\cup \{\infty\}$, with the singular point corresponding to 0 and $\infty$.
This explains why a singular point plus a regular point is well-defined: $0 \times a=0$ and $\infty \times a=\infty$ for $a\in \C^*$.
We can also see why a singular point plus a singular point can give any point on the singular fiber: $b \times a\slash b=a$ for $a,b\in \C^*$.
Taking $b\mapsto 0$ we get $0 \times \infty =a$.
This is consistent with the geometry of the induced group structure on $A_1$-singular fibers for elliptic fibrations \cite{Kod63}.
\null\hfill$\triangle$
\end{rmk}

\section{The addition graph}
\label{sec:addition_graph}

\subsection{$\bar\Gamma$ is an immersed Lagrangian correspondence}
\label{subsec:Gamma_is_immersed}

In light of \autoref{p3.5}, to encode the addition information we make the following definition.

\begin{defn}
Define $\Gamma_\sigma \subset (M \times M)^- \times M$ to be the addition graph with respect to the Lagrangian section $\sigma$, where it is well-defined:
\begin{equation}
    \Gamma_{\sigma}\coloneqq\{(x,y,z)\in M \times M \times M\;\;|\;\; \pi(x)=\pi(y)=\pi(z),(x+y)_{\sigma}=z\}.
\end{equation}
Denote its closure by $\bar{\Gamma}_{\sigma}$.
\null\hfill$\triangle$
\end{defn}

In the following, we will use $\Gamma$ (or $\bar{\Gamma}$) to mean the addition graph (or its closure) without specifying the section we choose.
For the model neighborhood $X_S$, we denote its addition graph (with respect to $\sigma_S$) by $\Gamma_S$ and the closure by $\bar{\Gamma}_S$.
We write $X^3_S$ for the triple Cartesian product $X_S\times X_S\times X_S$.

\begin{lem}
\label{lem:Gamma_S_is_immersed}
$\bar{\Gamma}_S$ is an immersed submanifold of $X_S^3$.
\end{lem}

\begin{proof}
We will directly build a smooth manifold $\widetilde{\Gamma}_S$ by defining its coordinate charts and transition maps, together with an immersion map $i_S$ onto $\bar{\Gamma}_S$.
For the sake of simplicity, we will assume $S\equiv 0$.
The same argument can be easily applied to an arbitrary $S$, see \autoref{r4.3}.
In order to describe the closure, we can express the addition with respect to $\sigma_1$ equivalently as  
\begin{equation}\label{e4.1}
    (a,b,c) \xmapsto{F_1} (\bar{a},-bc,\bar{b},-ac,\bar{a}\bar{b},-c)
\end{equation}
and the addition when choosing $\sigma_2$ as zero section
\begin{equation}\label{e4.2}
    (a,b,c) \xmapsto{F_2} (-\bar{b}\bar{c},a,-\bar{a}\bar{c},b,-\bar{c},ab).
\end{equation}

Without specifying the actual domain of $(a,b,c) \in \C^3$, by comparing with \autoref{e4.1} with \autoref{Add1}, and \autoref{e4.2} with \autoref{Add2} when $S \equiv 0$, we can see the image of $F_1, F_2$ agree with the graphs of the two addition maps on regular fibers (i.e.\ when $abc\neq 0$).
By setting $abc=0$, we can actually obtain the full induced map on the singular fiber, including where the addition map is not well-defined, i.e.\ $(s,s,F_s)$.
To see this, suppose $a=b=0$.
The image of $F_1$ gives $(0,0,0,0,0,-c)$, and by properly choosing the value of $c$ this gives a point lying in $(s,s,F_s)$ after identifying under quotient.
    
To build coordinate charts to cover $\bar{\Gamma}_S$, we need to consider open subsets to cover $X_S$ instead of $\D$.
By restricting each of the three points in the image of $F_1, F_2$ to lie inside $\D^0$ or $\U$, which covers our $X_S$ after taking the quotient, we can get a corresponding domain for $(a,b,c)$.
Restricting the map $F_1$ or $F_2$ to the above domains gives maps into $X'$ that descend to $X_S$ .
    
For example, if we restrict the image of $F_1$ to lie in $\D^0\times\D^0\times\D^0$,  we will have the corresponding domain $\{(a,b,c) \in D_1^3\;|\;|abc|<\epsilon\}$, where $D_1=\{z\in \C\;|\;|z|<1\}$ is the (open) unit disk.
Then the map $(\phi_0 \oplus \phi_0 \oplus \phi_0)\circ F_1$ will be an injective map into $\bar{\Gamma}_S$.
In the same way, we can obtain six embedding whose image we denote $N_i$, defined as follows:
\begin{align}
    E_1 \coloneqq \left\{(a,b,c) \in \C^3 \bigg\rvert|a|,|b|,|c||<1;|abc|<\epsilon\right\} & \xrightarrow{f_1 \coloneqq ( \phi_0 \oplus \phi_0 \oplus \phi_0)\circ F_1} \bar{\Gamma}_S, \\
    E_2 \coloneqq  \left\{(a,b,c) \in \C^3 \bigg\rvert |a|,|b|,|c||<1; |abc|<\epsilon\right\} & \xrightarrow{f_2 \coloneqq  (\phi_0 \oplus \phi_0 \oplus \phi_0)\circ F_2} \bar{\Gamma}_S, \nonumber \\
    E_3\coloneqq \left\{(a,b,c) \in \C^3 \bigg\rvert e^{-\delta}<|c|,|ac|,|bc|<e^{\delta};|abc|<\epsilon\right\} & \xrightarrow{f_3 \coloneqq  (\phi_1 \oplus \phi_1 \oplus \phi_1)\circ F_1} \bar{\Gamma}_S, \nonumber \\
    E_4\coloneqq \left\{(a,b,c) \in \C^3 \bigg\rvert e^{-\delta}<|b|<e^{\delta};|a|,|c|,|ab|,|bc|<1;|abc|<\epsilon\right\} & \xrightarrow{ f_4 \coloneqq (\phi_0 \oplus \phi_1 \oplus \phi_0)\circ F_1} \bar{\Gamma}_S, \nonumber \\
    E_5\coloneqq \left\{(a,b,c) \in \C^3 \bigg\rvert e^{-\delta}<|a|<e^{\delta};|b|,|c|,|ab|,|ac|<1;|abc|<\epsilon\right\} &  \xrightarrow{  f_5 \coloneqq (\phi_1 \oplus \phi_0 \oplus \phi_0)\circ F_1} \bar{\Gamma}_S, \nonumber \\
    E_6 \coloneqq  \left\{(a,b,c) \in \C^3 \bigg\rvert e^{-\delta}<|c|<e^{\delta};|b|,|a|,|bc|,|ac|<1;|abc|<\epsilon\right\}  & \xrightarrow{ f_6 \coloneqq (\phi_0 \oplus \phi_0 \oplus \phi_1)\circ F_2} \bar{\Gamma}_S. \nonumber
\end{align}

We claim that the manifolds $E_i$ for $1\leq i\leq 6$ can be glued together via transition maps to give a smooth manifold $\widetilde{\Gamma}_S$, and that the maps $f_i$ defined on each chart of $\widetilde{\Gamma}_S$ can be glued together to give an immersion $i_S\colon \widetilde{\Gamma}_S \rightarrow \bar{\Gamma}_S$ onto the image.
   
First, $F_1$ and $F_2$ are smooth injective maps.
By computation we can see that the Jacobian matrix of $F_1$ and $F_2$ are also injective.
Since $\D^0$ or $\U$ embeds into $X_S$, $f_i$ gives an embedding of an open subset of $\C^3$ into $X_S^3$ that lies in $\bar{\Gamma}_S$.

Second, $\bigcup_{i=1}^{6}N_i=\bar{\Gamma}_S$.
By the reasoning in \autoref{subsec:geometry_of_correspondence}, the image of $f_1$ and $f_2$ cover all of the points on $\bar{\Gamma}_S$ except for the ones that concern our zero section circles, i.e.\ $\Sigma_1\coloneqq\{(p,q)\in X\;|\;|q|=1\}$ or equivalently $\Sigma_2\coloneqq\{(p,q)\in X \;|\; |p|=e^{S_1(-\bar{p}q)}\}$.
The points are $(x,\sigma_S,x)$, $(\sigma_S,x,x)$, $(\sigma_S,\sigma_S,\sigma_S)$ and $(x,y,\sigma_S)$, where $x \in X_S\backslash \sigma_S$ and $x,y \in X_S$ have to satisfy the addition law $(x+y)_{\sigma_S}=s$.
These points are covered by the rest of the charts: $f_3$ covers $(\sigma_S,\sigma_S,\sigma_S)$, $f_4$ covers $(x,\sigma_S,x)$, $f_5$ covers $(\sigma_S,x,x)$ and $f_6$ covers $(x,y,\sigma_S)$.
Finally, note that $(s,s,s)\in \bar{\Gamma}_S$ is covered only by $f_1$ and $f_2$ and a calculation shows that this is the only common point in their image, i.e.\ $N_1 \cap N_2=(s,s,s)$.

Now, we define $\widetilde{\Gamma}_S$.
We will glue $E_i$ together by identifying the preimage of the same point in $\bar{\Gamma}_S$, except for $(s,s,s)$.
As a set, define $\widetilde{\Gamma}_S=\bar{\Gamma}_S \cup \{s'\}$, where $s'$ is a single point.
Define $N_i' \subset \widetilde{\Gamma}_S$ for $1\leq i\leq 6$ to be $N_1'=\{s'\}\cup N_1\backslash (s,s,s)$ and $N_i'=N_i$ for $i>1$.
Here we are identifying $\bar{\Gamma}_S$ to a subset of $\widetilde{\Gamma}_S$.
Note that $\bigcup_{i=1}^{6}N_i'=\widetilde{\Gamma}_S$ since $\bigcup_{i=1}^{6}N_i=\bar{\Gamma}_S$ and the origins in $E_1$ and $E_2$ both map to $(s,s,s)=N_1 \cap N_2 \subset \bar{\Gamma}_S$.
We have $N_1' \cap N_2'=\emptyset$ and $N_i' \cap N_j'=N_i\cap N_j$ otherwise.

Next we define coordinate maps $\psi_i\colon N_i'\rightarrow E_i \subset \R^6 (i=1,...6)$ to be $\psi_1(s')=(0,0,0) \in E_1$ and $\psi_i(x)=f_i^{-1}(x)$ otherwise. Again here we identifies $\bar{\Gamma}_S$ to a subset of $\widetilde{\Gamma}_S$.
Under these coordinate maps, we have that $\psi_i(N_i'\cap N_j')$ is always open. Whenever $N_i'\cap N_j' \neq \emptyset$, the transition map $\psi_i^{-1}\circ \psi_j\colon \psi_i(N_i'\cap N_j') \rightarrow \psi_j(N_i'\cap N_j')$ is always smooth, which can be easily verified.
These together give a smooth manifold structure on $\widetilde{\Gamma}_S$.

Finally, we define the immersion map $i_S\colon \widetilde{\Gamma}_S \rightarrow \bar{\Gamma}_S$ by $i_S\colon \bar{\Gamma}_S \subset \widetilde{\Gamma}_S \xrightarrow{id} \bar{\Gamma}_S$ and $i_S(s')=(s,s,s)$.
Note that this is equivalent to defining the map on each coordinate chart $E_i$ to be $f_i$.
This defines a smooth map that has injective derivative, thus an immersion.
Moreover, $i_S$ is injective on $\widetilde{\Gamma}_S \backslash i_S^{-1}(s,s,s)$, and $(s,s,s)$ is a double point.
Since $\text{Im}\;dF_1|_{(0,0,0)}+\text{Im}\;dF_2|_{(0,0,0)}=\C^6 \cong T_{(s,s,s)}X_S^3$, $\bar{\Gamma}_S$ has a transversal self-intersection point at $(s,s,s)$.
\end{proof}

\begin{rmk}\label{r4.3}
The above argument works for arbitrary $S$ in the similar way, as well as for a different choice of Lagrangian section.
For a different choice of $S$, we would use with \autoref{e4.1} modified with $S(abc)$ in its formula.
That said, the Jacobian stays injective.
Moreover, the charts will stay the same except for the domain: they will be defined using the definition of $\D^0$ and $\U$ with again $S$ involved.
But, the conditions stay open, so the domains are still open subsets of $\C^3$.
All the other reasoning works the same as above.

If we replace the Lagrangian section, then we would use the fact that, for a different section $\sigma_S'$, there exists $t(b)$ such that $\phi_{H}^{t(b)}(\sigma_S(b))=\sigma_S'(b)$.
Then a simple computation shows 
\begin{equation}
    (x+y)_{\sigma_S}=\phi_{H}^{t(b)}((x+y)_{\sigma_S'}).
\end{equation}
Denote the addition graph with respect to $\sigma_S'$ as $\Gamma_S'$.
The Hamiltonian flow $\phi_H^{t(b)}$ is a self-diffeomorphism on $X_S$, thus the map \begin{equation}
    (x,y,(x+y)_{\sigma_S})\in \bar{\Gamma}_S \mapsto \bigl(x,y,(x+y)_{\sigma_S'}=\phi_H^{t(b)}((x+y)_{\sigma_S})\bigr)\in \bar{\Gamma}_S'
\end{equation}
is a diffeomorphism onto the image when restricted on $\Gamma_S$.
Note that this map also sends $(s,s,s)$ to itself.
Then it is easy to see that the composition of the above map with $i_S$ gives an immersion map from $\widetilde{\Gamma}_S$ onto $\bar{\Gamma}_S'$.
\null\hfill$\triangle$
\end{rmk}

Now, based on the result for the local model, we have the following theorem.
In the statement of the theorem, $(M \times M)^- \times M$ is the symplectic manifold $M^3$ equipped with the symplectic form $\tilde{\omega}\coloneqq(-\omega)\oplus (-\omega)\oplus\omega$, which we assume is equipped with a singular torus fibration over $B$, with focus-focus singularities on distinct fibers.

\begin{thm}
\label{thm:Gamma-bar_is_immersed_Lagrangian}
$\bar{\Gamma}$ is an immersed Lagrangian submanifold of $(M \times M)^- \times M$ with transversal self-intersection double points at each triple $(s,s,s)$, where $s$ is a critical point.
\end{thm}

\begin{proof}
In order to produce an immersion $\widetilde\Gamma \hookrightarrow M^3$ whose image is $\bar\Gamma$, and which has a transversal double point at each $(s,s,s)$ and no other immersed points, we can combine \autoref{thm:focus-focus_classification} with \autoref{lem:Gamma_S_is_immersed}.

\cite[Lemma 5.0.33]{Sub10} shows that $\Gamma$ is Lagrangian.
Since being Lagrangian is a closed condition, $\bar{\Gamma}$ is also Lagrangian.
\end{proof}

\begin{rmk}
As $\bar{\Gamma}$ has isolated transversal double points and no other singularities, and the map from $\widetilde{\Gamma}$ is a diffeomorphism away from this point, we see that any immersion $i_N\colon N \rightarrow M^3$ with image $\bar{\Gamma}$ must factor through $\widetilde{\Gamma}$.
\null\hfill$\triangle$
\end{rmk}

\subsection{Geometry of the correspondence}
\label{subsec:geometry_of_correspondence}

In this section, we study the geometry of $\widetilde{\Gamma}$ and $\bar{\Gamma}$, starting with the observation that $M\times_B M$ and $\bar{\Gamma}$ are closely related.
Indeed, the map
\begin{equation}
    (x,y,z) \in \bar{\Gamma} \xmapsto{\pr} (x,y) \in M\times_B M
\end{equation}
is 1-1 everywhere except for pairs $(s, s)$, when $s$ is a critical points, whose preimage is the whole singular fiber $(s, s, F_s)$.

We first analyze $M\times_B M$.
This is a singular space, with isolated singularities $(s,s)$.
This can be seen by looking at the Jacobian of the defining equation $\{xy=zw\} \subset \C^4$ near the origin.
The fiber product $M\times_B M$ can also be seen as a singular fibration over $B$ by sending $(x,y)\in M \times_B M$ to $\pi(x)\in B$.
We denote this fibration map as $\pi'$.
This fibration is a regular fibration with fiber $T^2 \times T^2$ when taking out any singular fibers $F_s \times F_s$.
To see this, we can look at its tangent space.
Indeed, the tangent space of $(x,y) \in M \times_B M$ is $\{(v,w)\in T_xM \times T_yM\;|\;d\pi(v)=d\pi(w)\}$.
Thus $d\pi'|_{(x,y)}(v,w)=d\pi(v)=d\pi(w)$.
When $(x,y)$ are both regular, $v$ can be any vector in $T_xM$, thus $d\pi'$ is surjective.
When at least one of them is singular, e.g.\ at $(s,y)$, since $d\pi|_{s}$ is a zero map, so is $d\pi'|_{(s,y)}$.

The smooth manifold part of $M\times_B M$ embeds into $\widetilde{\Gamma}$ via the addition map. The complement of its image is $(\pr\circ i)^{-1}(s,s)\cong S^2$ if we denote the immersion map as $i:\widetilde{\Gamma} \rightarrow \bar{\Gamma}$.
The following lemma tells us what a neighborhood of this $S^2$ inside $\widetilde{\Gamma}$ looks like.

\begin{lem}
\label{lem:tubular_neighborhood}
A tubular neighborhood of $(\pr\circ i)^{-1}(s,s) \subset \widetilde{\Gamma}$ is diffeomorphic to a neighborhood of the zero section inside the total space of the bundle $O(-1)\oplus O(-1)$ over $\mathbb{P}^1$.
\end{lem}

\begin{proof}
It suffices to work with $X_S$.
As before we will prove in the following for the $S \equiv 0$ case for simplicity of formulas and a similar argument works for arbitrary $S$ with a slight modification to the formulas involved.

Recalling the charts we constructed for $\widetilde{\Gamma}_S$ in proof of \autoref{lem:Gamma_S_is_immersed}, the following charts can cover a neighborhood of the $S^2$:
\begin{itemize}
\item
Chart 1, using $\sigma_1$ as the zero section to describe addition on $\D^0\times\D^0\times\D^0$:
\begin{equation}
    E_1=\left\{(a,b,c) \in \C^3 \bigg\rvert|a|,|b|,|c||<1;|abc|<\epsilon\right\}
\end{equation}

\item
Chart 2, using $\sigma_2$ as the zero section to describe addition on $\D^0\times\D^0\times\D^0$:
\begin{equation}
    E_2=\left\{(a,b,c) \in \C^3 \bigg\rvert |a|,|b|,|c||<1; |abc|<\epsilon\right\}
\end{equation}

\item
Chart 6, using $\sigma_2$ as the zero section to describe addition on $\D^0\times\D^0\times\U$:
\begin{equation}
    E_6=\left\{(a,b,c) \in \C^3 \bigg\rvert e^{-\delta}<|c|<e^{\delta};|b|,|a|,|bc|,|ac|<1;|abc|<\epsilon\right\}
\end{equation}
\end{itemize}
 
The coordinate changes on the overlapping part are defined by a map $\phi\colon (a,b,c) \mapsto (-bc,-ac,1/c)$ from
\begin{equation}
    E_{16}=\left\{(a,b,c) \in D_1^3 \bigg\rvert |abc|<\epsilon ; |c|>e^{-\delta}\right\} \subset E_1
\end{equation}
to
\begin{equation}
    E_{61}=\left\{(a,b,c) \in D_1^2 \times \C \bigg\rvert 1<|c|<e^{\delta};|bc|,|ac|<1;|abc|<\epsilon\right\} \subset E_6
\end{equation}
and the identity map from
\begin{equation}
    E_{26}=\left\{(a,b,c) \in D_1^3 \bigg\rvert |abc|<\epsilon ; |c|>e^{-\delta}\right\} \subset E_2
\end{equation}
to
\begin{equation}
    E_{62}=\left\{(a,b,c) \in D_1^3 \bigg\rvert |abc|<\epsilon ; |c|>e^{-\delta}\right\} \subset E_6.
\end{equation}
Note that $\phi\circ \phi = \mathrm{id}$.
We denote the resulting manifold (i.e.\ a neighborhood of $S^2 \subset \widetilde{\Gamma}$ by diffeomorphism) by $U$.

In order to establish a diffeomorphism, we use the following description of the bundle $O(-1)\oplus O(-1)$ over $\mathbb{P}^1 $: over a point $\lambda \in \mathbb{P}^1$, by the definition of the tautological line bundle, points in the fiber can be described as $(c_1,c_1z,c_2,c_2z)$ for $(c_1,c_2)\in \C^2$ if $\lambda =[1:z]$ for some $z\in \C$ or $(c_1w,c_1,c_2w,c_2)$ for $(c_1,c_2)\in \C^2$ if $\lambda =[w:1]$ for some $w \in \C$.

We can now define maps $G_i$ for $i=1,2$, given by
\begin{align}
    G_1\colon & (a,b,c)\in \C^3 \mapsto ([1:-c],(b,-bc,a,-ac)) \in V, \\
    G_2\colon & (a,b,c) \in \C^3 \mapsto ([-c:1],(-ac,a,-bc,b)) \in V. \nonumber
\end{align}
We use these maps to define a map $G\colon U \longrightarrow V$ by gluing $E_1 \xrightarrow{G_1} V$, $E_2 \xrightarrow{G_2} V$ with $E_6 \xrightarrow{G_2} V$.
These maps glue to a well-defined map $G$ since they are compatible with transition maps.
For instance, we have
\begin{equation}
    (a,b,c)\in E_{16} \xmapsto{\phi} (-bc,-ac,1/c) \in E_{61}\xmapsto{G_2} ([-1/c,1],(b,-bc,a,-ac)),
\end{equation}
which is equal to
\begin{equation}
    ([1,-c],(b,-bc,a,-ac))\xmapsto{G_1}(a,b,c)\in E_{16}.
\end{equation}

It is easy to verify that the map $G$ is injective and thus it is a diffeomorphism onto the image.
Moreover, $G$ maps $S^2$ to the zero section in $V$.
By taking $a=b=0$ in the $E_i$ charts for $i=1,2,6$ we get the $S^2\in \widetilde{\Gamma}$, which gets mapped to points of the form $([1:-c],(0,0,0,0))$ or $([-c:1],(0,0,0,0))$, which lie in the zero section.
If we restrict $G$ on the $S^2$, we get a diffeomorphism onto the zero section.

In this way we construct a neighborhood of the $S^2$ that is diffeomorphic to a neighborhood of zero section of $O(-1)\oplus O(-1)$ with the $S^2$ corresponding to the zero section.
\end{proof}

\begin{rmk}
The description of the above local structure can be considered as a minimal resolution of the singularity of $M \times_B M$, which is consistent with the complex analytical picture.
Consider instead $\pi\colon M \rightarrow B$ as a elliptic fibration with $A_1$ singularities.
The local model for an $A_1$ singularity is $(x,y)\in \C^2 \mapsto xy\in \C$.
This makes the local model of $M \times_B M$ the singular variety $Y=\{(x,z,w,y)\in \C^4\;|\;xy=zw\}$.
The following discussion is due to \cite{Atiyah58}.
    
To resolve the singular point at the origin, we can first do the usual blow up: consider the projectivization of $Y$, $V_2\coloneqq\{[x:z:w:y]\in \mathbb{P}^3\;|\;xy=zw\}$.
It is easy to see that $V_2 \cong \mathbb{P}^1 \times \mathbb{P}^1$ by $(\lambda,\lambda')\in \mathbb{P}^1\times \mathbb{P}^1 \mapsto [\lambda'\lambda:\lambda':\lambda:1] \in \mathbb{P}^3$.
Blowing up gives us a smooth variety $\tilde{Y}=\{(\lambda,\lambda',c)\in \mathbb{P}^1\times \mathbb{P}^1 \times \C^4 \;|\; \exists s\in \C:
c=(s\lambda'\lambda,s\lambda',s\lambda,s) \}$.
       
$\tilde{Y}$ can be further blown down: we can choose either to blow down the first $\mathbb{P}^1_{\lambda}$-factor or the second $\mathbb{P}^1_{\lambda'}$-factor, denoting the resulting varieties as $E_{\lambda'}$ and $E_{\lambda}$.
They are both smooth complex manifolds, especially $E_{\lambda}\cong O(-1)\oplus O(-1)$.

To see this, consider the natural projection of $(\lambda,c)\in E_{\lambda}\mapsto \lambda \in \mathbb{P}^1$.
When $\lambda \neq \infty$, we can represent $\lambda$ as $[w:1]$ for some $w\in \C$, then the preimage (under the projection map) of $\lambda$ is $(aw,a,bw,b)$ for $(a,b)\in \C^2$.
When $\lambda \neq 0$, we can represent $\lambda$ as $[1:z]$ for some $z\in \C$, then the preimage (under the projection map) of $\lambda$ is $(a,az,b,bz)$ for $(a,b)\in \C^2$.
This gives $E_{\lambda}\cong O(-1)\oplus O(-1)$.
       
They can be both seen as minimal resolutions of $Y$ by replacing the singular point with the exceptional curve $\mathbb{P}^1$.
$E_{\lambda}$ and $E_{\lambda'}$ are not isomorphic as complex manifolds but the normal bundle to the exceptional $\mathbb{P}^1$ are both $O(-1)\oplus O(-1)$.
\null\hfill$\triangle$
\end{rmk}

\bibliographystyle{alpha}
\bibliography{bib}

%----------------------------------------------------------------------------------------%

\end{document}